\documentclass[twoside,11pt]{article}
\usepackage{amssymb,latexsym,amsmath,amsthm}
\usepackage[numbers,sort&compress]{natbib}
\usepackage{rotating}
\usepackage{multirow}
\begin{document}
\setlength{\unitlength}{10mm}
\newcommand{\f}{\frac}
\newtheorem{theorem}{Theorem}
\newtheorem{lemma}{Lemma}
\newtheorem{remark}{Remark}
\newtheorem{definition}{Definition}
\newcommand{\sta}{\stackrel}
\title{Interval extension of the three-step Kung and Traub's method}
\author{
Tahereh Eftekhari
\\
{\scriptsize School of Mathematics, Iran University of Science $ \& $ Technology (IUST), Narmak, Tehran 16846 13114, Iran}
\\
{\scriptsize t.eftekhari2009@gmail.com}
}

\date{}
\maketitle

\begin{abstract}
\noindent In this paper we produce an interval extension of the three-step Kung and Traub's method for solving nonlinear equations. Furthermore, the convergence analysis of the new method is discussed and this method is compared to already present methods.

\vspace{.5cm}\leftline{\textbf{Keywords: }Interval analysis, Nonlinear equations, Kung and Traub's method, Guaranteed convergence.}
\vspace{.5cm}\leftline{\textbf{Mathematics Subject Classification 2010: }65H05.}
\end{abstract}

\section{Introduction}\label{sec:1}
Solving nonlinear equations is one of the basic problems in scientific and engineering applications. For this purpose, interval methods has been developed for finding the enclosure solutions of nonlinear equations, see \cite{E1,E2,MKC,P,SG}. These methods can be used to refine enclosures to solutions of nonlinear equations, to prove existence and uniqueness of such solutions, and to provide rigorous bounds on such solutions. Interval methods can also prove non-existence of solutions within regions.
We first introduce some basic properties of interval arithmetic from \cite{M1,MKC}. A real interval is a closed connected subset of $ \mathbb{R} $, which is of the form
\[X = [\underline {x} , \overline {x} ] = \{ x\in \mathbb{R}\, | \,\, \underline {x} \leq x \leq \overline {x} \},\] where $ \underline {x} $ and $ \overline {x} $ represent, respectively, the lower and the upper bounds of $ X $. Intervals with $ \underline {x} = \overline {x} $ are called thin, point or degenerate intervals, while intervals with $ \underline {x}\leq \overline {x} $
are called thick or proper intervals. The set of all closed real intervals is denoted by $ \mathbb{IR} $. The width, radius, mid-point and absolute value of an interval $ X = [\underline {x} , \, \overline {x} ] $ are respectively defined as
\[\begin{array}{l}
w(X) = {\overline {x}} - {\underline {x}}, \\
rad (X) = \dfrac{1}{2}({\overline {x}} - {\underline {x}}), \\
m(X) = \dfrac{1}{2}({\underline {x}} + {\overline {x}}), \\
|X| = \max \{ |\underline x |,|\overline x |\}.
\end{array}\]
For $ X = [\underline x ,\overline x ] $ and $ Y = [\underline y ,\overline y ] $, $ X\oplus Y $ with $ \oplus\in \{ +, -, ., /\} $ is defined by $ X\oplus Y = \{ x\oplus y \, |\, x\in X,\, y\in Y \} $. We have
\[\begin{array}{l}
X+Y = [\underline x + \underline y ,\overline x + \overline y ],\\
X-Y = [\underline x - \overline y ,\overline x - \underline y ], \\
X.Y = [ \min\{ \underline x \underline y , \underline x \overline y , \overline x \underline y , \overline x \overline y \}, \max \{ \underline x \underline y , \underline x \overline y , \overline x \underline y , \overline x \overline y \} ], \\
X/ Y = X. (1/Y), \\
1/Y = [ 1/\overline y , 1/\underline y ], \quad 0\notin Y.
\end{array}\]
Note that subtraction and division are not the inverse operations of addition and respectively multiplication. An interval $ X $ is a subset of an interval $ Y, $ is denoted by $ X\subseteq Y, $ if and only if $ \underline y \leq \underline x $ and $ \overline y \geq \overline x $. The intersection of two intervals $ X $ and $ Y $, denoted $ X\cap Y $, is the set of all elements that belong to both $ X $ and $ Y $. That is,
\[ X\cap Y = \{ z \, |\, z\in X , z\in Y \} = [\max\{\underline x , \underline y\}, \min\{\overline x , \overline y\}]. \]
The intersection of two intervals $ X $ and $ Y $ is empty if either $ \overline y < \underline x $ or $ \overline x < \underline y $, in this case we write $ X\cap Y = \emptyset $.


\begin{definition}\label{f1-1}
We say that $ F $ is an interval extension of $ f $ on the interval $ X = [\underline x ,\overline x ] $, if
\[\begin{array}{l}
F([x, x]) = f (x),\quad (restriction),\\
F(X) \supseteq \{f (x)\, |\, x\in X \},\quad (inclusion).
\end{array}\]
\end{definition}


\begin{definition}\label{f1}
$ F $ is Lipschitz' interval extension of $ f $ in an interval $ {X^{(0)}} $ if there is a constant $ L $ such that $ w(F(X)) \leq L\, w(X) $ for every $ X \subseteq {X^{(0)}} $.
\end{definition}


\begin{definition}\label{f1-2}
An interval valued function $ F $ is inclusion monotonic if $ X\subseteq Y $ implies $ F(X)\subseteq F(Y) $.
\end{definition}


\begin{definition}\label{f2}
An interval sequence $ {X^{(k)}} $ is nested if $ {X^{(k+1)}} \subseteq {X^{(k)}} $ for all $ k. $
\end{definition}


\begin{lemma}[see \cite{MKC}]\label{f3}
Suppose $ \{ {{X^{(k)}}} \} $ is such that there is a real number $ x\in {X^{(k)}} $ for all $ k. $ Define
 $ \{ {{Y^{(k)}}} \} $ by $ {Y^{(1)}} = {X^{(1)}} $ and  $ {Y^{(k+1)}} = {X^{(k+1)}} \cap {Y^{(k)}} $ for all $ k = 1, 2,\cdots. $ Then
$ {Y^{(k)}} $ is nested with limit $ Y, $ and \[ x \in Y \subseteq {Y^{(k)}},\qquad \forall k. \]
\end{lemma}


\begin{lemma}[see \cite{MKC}]\label{f4}
Every nested sequence $ {X^{(k)}} $ converges and has the limit $ \bigcap\limits_{k = 0}^\infty  {{X^{(k)}}}. $ 
\end{lemma}

Let $ f $ be a real-valued function of a real variable $ x, $ and suppose that $ f $ is continuously differentiable. Newton's method is one of the best iterative methods for solving nonlinear equations by using
\begin{equation}\label{1-1}
{x_{n+1}} = {x_n} - \dfrac{{f({x_n})}}{{f'({x_n})}},
\end{equation}
which converges quadratically (see \cite{O}).

Let $ F'( {X} ) $ be an inclusion monotonic interval extension of $ f'( {x} ) $ such that $ 0\notin F'(X) $. An interval version of Newton method has been developed for solving nonlinear equations in \cite{MKC} as follows:
\begin{equation}\label{1}
{X^{(k+1)}} = N({X^{(k)}}) \cap {X^{(k)}},\qquad k = 0, 1, 2, \cdots ,\\
\end{equation}
where
\[N({X^{(k)}}) = {m({X^{(k)}}) - \dfrac{{f(m({X^{(k)}}))}}{{F'({X^{(k)}})}}}. \]


\begin{theorem}[see \cite{MKC}]\label{f4-1}
If an interval $ {X^{(0)}} $ contains a zero $ x $ of $ f (x) $, then so does $ {X^{(k)}} $ for all $ k=0, 1, 2, \cdots $, defined by \eqref{1}. Furthermore, the intervals $ {X^{(k)}} $ form a nested sequence converging to $ x $ if $ 0 \notin F'( {{X^{(0)}}} ) $.
\end{theorem}


\begin{theorem}[see \cite{MKC}]\label{f5}
Given a real rational function $ f $ of a single real variable $ x $ with rational extensions $ F,\, F' $ of $ f,\, f', $ respectively, such that $ f $ has a simple zero $ x^* $ in an interval $ {X^{(0)}} $ for which $ F( {{X^{(0)}}} ) $ is defined and $ F'( {{X^{(0)}}} ) $ is defined and does not contain zero i.e. $ 0 \notin F'( {{X^{(0)}}} ). $ Then there is a positive real number $ C $ such that
\[w( {{X^{( k + 1 )}}} ) \leq C\, (w( {{X^{( k )}}} ))^2.\]
\end{theorem}


\begin{remark}\label{f5-2}
Let $ k=0, 1, 2, \cdots $. The interval Newton method has the following properties:
\begin{itemize}
\item[$ (i) $] If $ N({X^{(k)}}) \cap {X^{(k)}}=\emptyset , $ then $ {X^{(k)}} $ does not contain any zero of $ f $.
\item[$ (ii) $] If $ x^* \in {X^{(0)}} $ and $ N({X^{(k)}}) \subseteq {X^{(k)}}, $ then $ {X^{(k)}} $ contains exactly one zero of $ f $.
\end{itemize}
\end{remark}


Recently, based on interval extension of the Newton method \eqref{1}, some interval methods have been produced for computing the enclosure solutions of nonlinear equations. In \cite{E1}, an interval extension of the King method have been produced as follows:
\begin{equation}\label{2-1}
\begin{cases}
{Y^{(k)}} = \left\{ {m ({X^{(k)}}) - \dfrac{{f(m ({X^{(k)}}))}}{{F'({X^{(k)}})}}} \right\} \cap {X^{(k)}},\\
{X^{(k+1)}} = \left\{ m ({Y^{(k)}}) - \left(\dfrac{f(m ({X^{(k)}})) +\beta f(m ({Y^{(k)}}))}{f(m ({X^{(k)}})) + (\beta -2) f(m ({Y^{(k)}}))}\right)\dfrac{f(m ({Y^{(k)}}))}{F'({X^{(k)}})} \right\} \cap {X^{(k)}},
\end{cases}
\end{equation}
where $ \beta\in \mathbb{R} $ is a constant. Interval extension of the Ostrowski method \cite{E2} is a member of this family when $ \beta = 0 $. In \cite{P}, an interval version of Traub's three-step method has been produced by Petkovi\'c, which is written as:
\begin{equation}\label{2-1-1}
\begin{cases}
\begin{array}{l}
{Y^{(k)}} = \left\{ {m({X^{(k)}}) - \dfrac{{f(m({X^{(k)}}))}}{{F'({X^{(k)}})}}} \right\} \cap {X^{(k)}},\\
{Z^{(k)}} = \left\{ {m({Y^{(k)}}) - \dfrac{{f(m({Y^{(k)}}))}}{{F'({X^{(k)}})}}} \right\} \cap {Y^{(k)}},\\
{X^{(k + 1)}} = \left\{ {m({Z^{(k)}}) - \dfrac{{f(m({Z^{(k)}}))}}{{F'({X^{(k)}})}}} \right\} \cap {Z^{(k)}}.
\end{array}
\end{cases}
\end{equation}
We here remind the three-step method, which was given by Kung and Traub \cite{KT} as comes next
\begin{equation}\label{2-2}
\begin{cases}
{y_k} = {x_k} - \dfrac{{f({x_k})}}{f'({x_k})}, \\
{z_k} = {y_k} - \dfrac{{f({x_k})f({y_k})}}{{{{(f({x_k}) - f({y_k}))}^2}}}\dfrac{{f({x_k})}}{{f'({x_k})}}, \\
{x_{k + 1}} = {z_k} - \dfrac{{f({x_k})f({y_k})f({z_k})\left( {f{{({x_k})}^2} + f({y_k})(f({y_k}) - f({z_k}))} \right)}}{{{{(f({x_k}) - f({y_k}))}^2}{{(f({x_k}) - f({z_k}))}^2}(f({y_k}) - f({z_k}))}}\dfrac{{f({x_k})}}{{f'({x_k})}}.
\end{cases}
\end{equation}
This method is an improvement of Newton method \eqref{1-1} with the order of convergence equal to 8.

Let $x_{0}\leq x_{1}\leq \ldots \leq x_{n}$ and $ f $ is sufficiently differentiable. Here we write the inverse interpolation polynomial of degree at most $ n $, in the form
\begin{equation}\label{3-0}
\begin{array}{l}
{Q_n}(t) = {x_0} + (t - f({x_0})) g[f({x_0}),f({x_1})] +  \cdots \\
\,~~~~~~~~~~~~~~ + (t - f({x_0})) \cdots (t - f({x_{n - 1}})) g[f({x_0}),f({x_1}), \ldots ,f({x_n})],
\end{array}
\end{equation}
where
\[g(f(x)) = x,\]
and its derivative is
\[g'(f(x)) = 1/f'(x). \]
The inverse differences for $ {x_0} \neq {x_n} $ are defined recursively as follows:
\begin{equation}\label{3-1}
\begin{array}{l}
g[f({x_0}),f({x_1})] = \dfrac{{g(f({x_1})) - g(f({x_0}))}}{{f({x_1}) - f({x_0})}} = \dfrac{{{x_1} - {x_0}}}{{f({x_1}) - f({x_0})}},\\
\vdots \\
g[f({x_0}),f({x_1}), \ldots ,f({x_n})] = \dfrac{{g[f({x_1}), f({x_2}), \ldots ,f({x_n})] - g[f({x_0}), f({x_1}), \ldots ,f({x_{n - 1}})]}}{{f({x_n}) - f({x_0})}},
\end{array}
\end{equation}
if $ {x_0} = {x_n} $, then
\[ g[f({x_0}),f({x_1}), \ldots ,f({x_n})]= \dfrac{{{g^{(n)}}(f({x_0}))}}{{n!}}. \]
Putting $ t = 0 $ into \eqref{3-0} we get $ {x^*}\approx {Q_n}(0) $, where $ {x^*} $ is a root of $ f $.

In the present article, using the interval extension of the Newton method and applying the inverse interpolation polynomial, interval extension of the three-step Kung and Traub's method is produced for finding the root enclosures of nonlinear equations. Convergence rate of the proposed method is also examined. Moreover, error bound and comparison of this method with the already present methods are given.


\section{Interval three-step Kung and Traub's method}\label{sec:2}


Let $ {Y^{(0)}} = [ {y_1^{(0)} , y_2^{(0)}} ] $ be an interval such that $ {x^*}\approx m({X^{(0)}}) $ and $ {x^*}\approx m({Y^{(0)}}) $. From \eqref{1} we have
\begin{equation}\label{2}
{Y^{(0)}} = N({X^{(0)}}) \cap {X^{(0)}},
\end{equation}
where
\begin{equation}\label{2-0}
N({X^{(0)}}) = {m ({X^{(0)}}) - \dfrac{{f(m ({X^{(0)}}))}}{{F'({X^{(0)}})}}}.
\end{equation}
Since $ m ({Y^{(0)}}) \in {Y^{(0)}} $ and ${Y^{(0)}} \subseteq N({X^{(0)}}),$ it is clear that $ m ({Y^{(0)}}) \in N({X^{(0)}}). $ Suppose 
\[ m({Y^{(0)}}) = {m({X^{(0)}}) - \dfrac{{f(m({X^{(0)}}))}}{{f'(\alpha)}}}, \]
where $ f'(\alpha)\in F'({X^{(0)}}) $. Let $ \alpha $ is sufficiently close to $ m({X^{(0)}}) $. Since $ f $ is a smooth function, we can assume that $f'(\alpha) \approx f'(m({X^{(0)}})). $ Therefore
\begin{equation}\label{2-3}
m({Y^{(0)}}) \approx {m({X^{(0)}}) - \dfrac{{f(m({X^{(0)}}))}}{{f'(m({X^{(0)}}))}}}.
\end{equation}
We assume that near the root $ x^* $, the function $ f $ is monotone, so that $ f $ has an inverse $ g $. Since we have three values $ {f(m({X^{(0)}}))}, {f'(m({X^{(0)}}))} $ and $ {f(m({Y^{(0)}}))} $, it is convenient to approximate $ x^* $ by the inverse Hermite interpolation polynomial of degree 2. Hence, we can start a table of divided differences for the inverse function $ g $:
\begin{center} 
\resizebox{\textwidth}{!}{
\begin{tabular}{c|c|cc}
 $ f $ & $ g $ &  & \\
\hline
    $ f(m({X^{(0)}})) $ & $ m({X^{(0)}}) $ & \\
    
    $ f(m({X^{(0)}})) $ & $ m({X^{(0)}}) $ & $ g[f(m({X^{(0)}})), f(m({X^{(0)}}))] $ & \\
    
    $ f(m({Y^{(0)}})) $ & $ m({Y^{(0)}}) $ & $ g[f(m({X^{(0)}})), f(m({Y^{(0)}}))] $ & $ g[f(m({X^{(0)}})), f(m({X^{(0)}})), f(m({Y^{(0)}}))] $

\end{tabular}}
\end{center}
Let $ {{Z^{(0)}} = [z_1^{(0)},z_2^{(0)}] \subseteq {Y^{(0)}}} $ such that $ {x^*}\approx m({Z^{(0)}}) $. Wanting to compute $ {x^*} $, we can get an improved approximation by quadratic interpolation,
\begin{equation}\label{3-4}
\begin{scriptsize}
\begin{array}{l}
m({Z^{(0)}}) \approx m({X^{(0)}}) + (0 - f(m({X^{(0)}})))\, g[f(m({X^{(0)}})),f(m({X^{(0)}}))] \\
~~~~~~~~~~~~~~~~~~~~~~~~~~ + {(0 - f(m({X^{(0)}})))^2}\, g[f(m({X^{(0)}})),f(m({X^{(0)}})),f(m({Y^{(0)}}))]\\
~~~~~~~~~~~ = m({X^{(0)}}) - f(m({X^{(0)}}))\, g[f(m({X^{(0)}})),f(m({X^{(0)}}))]+ {f^2}(m({X^{(0)}}))\, g[f(m({X^{(0)}})),f(m({X^{(0)}})),f(m({Y^{(0)}}))],
\end{array}
\end{scriptsize}
\end{equation}
where
\begin{equation}\label{4-3}
g[f(m({X^{(0)}})),f(m({X^{(0)}}))] = g'(f(m({X^{(0)}}))) = \dfrac{1}{{f'(m({X^{(0)}}))}},
\end{equation}
and from \eqref{2-3}, we have
\begin{equation}\label{4-4}
\begin{scriptsize}
\begin{array}{l}
g[m({X^{(0)}}),m({X^{(0)}}),m({Y^{(0)}})] = \dfrac{{g[f(m({X^{(0)}})),f(m({Y^{(0)}}))] - g[f(m({X^{(0)}})),f(m({X^{(0)}}))]}}{{f(m({Y^{(0)}})) - f(m({X^{(0)}}))}}\\
~~~~~~~~~~~~~~~~~~~~~~~~~~~~~~~~~~~~~~~ = \dfrac{{\left( {(m({Y^{(0)}}) - m({X^{(0)}}))/(f(m({Y^{(0)}})) - f(m({X^{(0)}})))} \right) - \left( {1/f'(m({X^{(0)}}))} \right)}}{{f(m({Y^{(0)}})) - f(m({X^{(0)}}))}}\\
~~~~~~~~~~~~~~~~~~~~~~~~~~~~~~~~~~~~~~~ \approx \dfrac{{ - \left( {f(m({X^{(0)}}))/(f(m({Y^{(0)}})) - f(m({X^{(0)}})))f'(m({X^{(0)}}))} \right) - \left( {1/f'(m({X^{(0)}}))} \right)}}{{f(m({Y^{(0)}})) - f(m({X^{(0)}}))}}\\
~~~~~~~~~~~~~~~~~~~~~~~~~~~~~~~~~~~~~~~ = \dfrac{{ - f(m({Y^{(0)}}))}}{{{{(f(m({X^{(0)}})) - f(m({Y^{(0)}})))}^2}f'(m({X^{(0)}}))}}.
\end{array}
\end{scriptsize}
\end{equation}
Substituting \eqref{4-3} and \eqref{4-4} in \eqref{3-4}, we get
\begin{equation}\label{4-5}
\begin{array}{l}
m({Z^{(0)}}) \approx m({X^{(0)}}) - \dfrac{{f(m({X^{(0)}}))}}{{f'(m({X^{(0)}}))}} - \dfrac{{{f^2}(m({X^{(0)}}))f(m({Y^{(0)}}))}}{{{{(f(m({X^{(0)}})) - f(m({Y^{(0)}})))}^2}f'(m({X^{(0)}}))}}\\
~~~~~~~~~~\approx m({Y^{(0)}}) - \dfrac{{f(m({X^{(0)}}))f(m({Y^{(0)}}))}}{{{{(f(m({X^{(0)}})) - f(m({Y^{(0)}})))}^2}}}\dfrac{{f(m({X^{(0)}}))}}{{f'(m({X^{(0)}}))}}.
\end{array}
\end{equation}
Note that $ f'(m({X^{(0)}}))\in F'({X^{(0)}}) $, so the above formula yields
\[m({Z^{(0)}})\in \left\{ {m({Y^{(0)}}) - \dfrac{{f(m({X^{(0)}}))f(m({Y^{(0)}}))}}{{{{(f(m({X^{(0)}})) - f(m({Y^{(0)}})))}^2}}}\dfrac{{f(m({X^{(0)}}))}}{{F'({X^{(0)}})}}} \right\}.\]
Let
\begin{equation}\label{4-6}
K({X^{(0)}}, {Y^{(0)}}) = {m({Y^{(0)}}) - \dfrac{{f(m({X^{(0)}}))f(m({Y^{(0)}}))}}{{{{(f(m({X^{(0)}})) - f(m({Y^{(0)}})))}^2}}}\dfrac{{f(m({X^{(0)}}))}}{{F'({X^{(0)}})}}},
\end{equation}
therefore $ m({Z^{(0)}})\in K({X^{(0)}}, {Y^{(0)}}) $. Since $ m({Z^{(0)}}) \in Z^{(0)} \subseteq Y^{(0)} $, we have
\[\begin{array}{l}
m({Z^{(0)}}) \in K({X^{(0)}}, {Y^{(0)}}) \cap Y^{(0)}.
\end{array}\]
Define
\begin{equation}\label{4-7}
{Z^{(0)}} = K({X^{(0)}}, {Y^{(0)}}) \cap Y^{(0)}.
\end{equation}
Now we have $ m({Z^{(0)}}) = g(f(m({Z^{(0)}}))) $. Hence, the table of divided differences can be updated and becomes
\begin{center} 
\resizebox{\textwidth}{!}{
\begin{tabular}{c|c|ccc}
 $ f $ & $ g $ &  &  & \\
\hline
    $ f(m({X^{(0)}})) $ & $ m({X^{(0)}}) $ &  & \\
    
    $ f(m({X^{(0)}})) $ & $ m({X^{(0)}}) $ & $ g[f(m({X^{(0)}})), f(m({X^{(0)}}))] $ &  & \\
    
    $ f(m({Y^{(0)}})) $ & $ m({Y^{(0)}}) $ & $ g[f(m({X^{(0)}})), f(m({Y^{(0)}}))] $ & $ g[f(m({X^{(0)}})), f(m({X^{(0)}})), f(m({Y^{(0)}}))] $  &  \\
    
    $ f(m({Z^{(0)}})) $ & $ m({Z^{(0)}}) $ & $ g[f(m({Y^{(0)}})), f(m({Z^{(0)}}))] $ & $ g[f(m({X^{(0)}})), f(m({Y^{(0)}})), f(m({Z^{(0)}}))] $  & $ g[f(m({X^{(0)}})), f(m({X^{(0)}})), f(m({Y^{(0)}})), f(m({Z^{(0)}}))] $

\end{tabular}}
\end{center}
Let $ {{X^{(1)}} \subseteq {Z^{(0)}}} $ such that $ {x^*} \approx m({X^{(1)}}) $. This allows us to use cubic interpolation polynomial to get, again with  inverse interpolation polynomial,
\begin{equation}\label{4-10}
\begin{scriptsize}
\begin{array}{l}
m({X^{(1)}}) \approx m({X^{(0)}}) + (0 - f(m({X^{(0)}})))g[f(m({X^{(0)}})), f(m({X^{(0)}}))]\\
~~~~~~~~~~~~~~~~~~~~~~~~~~\, + {(0 - f(m({X^{(0)}})))^2}\, g[f(m({X^{(0)}})), f(m({X^{(0)}})), f(m ({Y^{(0)}}))]\\
~~~~~~~~~~~~~~~~~~~~~~~~~~\, + {(0 - f(m({X^{(0)}})))^2}(0 - f(m({Y^{(0)}})))\, g[f(m({X^{(0)}})), f(m({X^{(0)}})), f(m ({Y^{(0)}})), f(m ({Z^{(0)}}))]\\
~~~~~~~~~~~\, \approx m ({Z^{(0)}}) - {f^2}(m({X^{(0)}}))f(m({Y^{(0)}}))g[f(m({X^{(0)}})),f(m({X^{(0)}})),f(m({Y^{(0)}})),f(m({Z^{(0)}}))],
\end{array}
\end{scriptsize}
\end{equation}
to compute $ g[f(m({X^{(0)}})),f(m({X^{(0)}})),f(m({Y^{(0)}})),f(m({Z^{(0)}}))] $, first note that from \eqref{2-3} and \eqref{4-5}, we have
\begin{equation}\label{4-11}
\begin{small}
\begin{array}{l}
g[f(m({X^{(0)}})),f(m({Y^{(0)}})),f(m({Z^{(0)}}))]\\
 = \dfrac{{g[f(m({Y^{(0)}})),f(m({Z^{(0)}}))] - g[f(m({X^{(0)}})),f(m({Y^{(0)}}))]}}{{f(m({Z^{(0)}})) - f(m({X^{(0)}}))}}\\
 = \dfrac{{\left( {(m({Z^{(0)}}) - m({Y^{(0)}}))/(f(m({Z^{(0)}})) - f(m({Y^{(0)}})))} \right) - \left( {(m({Y^{(0)}}) - m({X^{(0)}}))/(f(m({Y^{(0)}})) - f(m({X^{(0)}})))} \right)}}{{f(m({Z^{(0)}})) - f(m({X^{(0)}}))}}\\
 = \dfrac{{(m({Z^{(0)}}) - m({Y^{(0)}}))(f(m({Y^{(0)}})) - f(m({X^{(0)}}))) - (m({Y^{(0)}}) - m({X^{(0)}}))(f(m({Z^{(0)}})) - f(m({Y^{(0)}})))}}{{(f(m({Z^{(0)}})) - f(m({X^{(0)}})))(f(m({Y^{(0)}})) - f(m({X^{(0)}})))(f(m({Z^{(0)}})) - f(m({Y^{(0)}})))}}\\
 \approx \dfrac{{ - \left( {f(m({X^{(0)}}))f(m({Y^{(0)}}))/(f(m({Y^{(0)}})) - f(m({X^{(0)}})))} \right) + \left( {f(m({Z^{(0)}})) - f(m({Y^{(0)}}))} \right)}}{{(f(m({Z^{(0)}})) - f(m({X^{(0)}})))(f(m({Y^{(0)}})) - f(m({X^{(0)}})))(f(m({Z^{(0)}})) - f(m({Y^{(0)}})))}}\dfrac{{f(m({X^{(0)}}))}}{{f'(m({X^{(0)}}))}}\\
 = \dfrac{{ - f(m({X^{(0)}}))f(m({Y^{(0)}})) + (f(m({Z^{(0)}})) - f(m({Y^{(0)}})))(f(m({Y^{(0)}})) - f(m({X^{(0)}})))}}{{(f(m({Z^{(0)}})) - f(m({X^{(0)}}))){{(f(m({X^{(0)}})) - f(m({Y^{(0)}})))}^2}(f(m({Z^{(0)}})) - f(m({Y^{(0)}})))}}\dfrac{{f(m({X^{(0)}}))}}{{f'(m({X^{(0)}}))}}\\
 = \dfrac{{f(m({X^{(0)}}))f(m({Y^{(0)}})) - (f(m({Z^{(0)}})) - f(m({Y^{(0)}})))(f(m({Y^{(0)}})) - f(m({X^{(0)}})))}}{{(f(m({Z^{(0)}})) - f(m({X^{(0)}}))){{(f(m({X^{(0)}})) - f(m({Y^{(0)}})))}^2}(f(m({Y^{(0)}})) - f(m({Z^{(0)}})))}}\dfrac{{f(m({X^{(0)}}))}}{{f'(m({X^{(0)}}))}}.
\end{array}
\end{small}
\end{equation}
Now from \eqref{4-4} and \eqref{4-11}, we obtain
\begin{equation}\label{4-12}
\begin{small}
\begin{array}{l}
g[m({X^{(0)}}),m({X^{(0)}}),m({Y^{(0)}}),m({Z^{(0)}})]\\
 = \dfrac{{g[f(m({X^{(0)}})),f(m({Y^{(0)}})),f(m({Z^{(0)}}))] - g[f(m({X^{(0)}})),f(m({X^{(0)}})),f(m({Y^{(0)}}))]}}{{f(m({Z^{(0)}})) - f(m({X^{(0)}}))}}\\
 \approx \left( {\dfrac{{f(m({X^{(0)}}))f(m({Y^{(0)}})) - (f(m({Y^{(0)}})) - f(m({Z^{(0)}})))(f(m({X^{(0)}})) - f(m({Y^{(0)}})))}}{{(f(m({Z^{(0)}})) - f(m({X^{(0)}}))){{(f(m({X^{(0)}})) - f(m({Y^{(0)}})))}^2}(f(m({Y^{(0)}})) - f(m({Z^{(0)}})))}}\dfrac{{f(m({X^{(0)}}))}}{{f'(m({X^{(0)}}))}}} \right.\\
~~\left. { + \dfrac{{f(m({Y^{(0)}}))}}{{{{(f(m({X^{(0)}})) - f(m({Y^{(0)}})))}^2}}}\dfrac{1}{{f'(m({X^{(0)}}))}}} \right)/(f(m({Z^{(0)}})) - f(m({X^{(0)}})))\\
 = \left( {\left( {f(m({X^{(0)}}))f(m({Y^{(0)}})) - (f(m({Y^{(0)}})) - f(m({Z^{(0)}})))(f(m({X^{(0)}})) - f(m({Y^{(0)}})))} \right)f(m({X^{(0)}}))} \right.\\
~~\left. { + f(m({Y^{(0)}}))(f(m({Z^{(0)}})) - f(m({X^{(0)}})))(f(m({Y^{(0)}})) - f(m({Z^{(0)}})))} \right)\\
~~~/{(f(m({Z^{(0)}})) - f(m({X^{(0)}})))^2}{(f(m({X^{(0)}})) - f(m({Y^{(0)}})))^2}(f(m({Y^{(0)}})) - f(m({Z^{(0)}})))f'(m({X^{(0)}}))\\
 = \dfrac{{\left( {{f^2}(m({X^{(0)}})) + f(m({Y^{(0)}}))(f(m({Y^{(0)}})) - f(m({Z^{(0)}})))} \right)f(m({Z^{(0)}}))}}{{{{(f(m({X^{(0)}})) - f(m({Z^{(0)}})))}^2}{{(f(m({X^{(0)}})) - f(m({Y^{(0)}})))}^2}(f(m({Y^{(0)}})) - f(m({Z^{(0)}})))f'(m({X^{(0)}}))}}.
\end{array}
\end{small}
\end{equation}
Substituting \eqref{4-12} into \eqref{4-10} gives
\begin{scriptsize}
\[\begin{array}{l}
m({X^{(1)}}) \approx m({Z^{(0)}}) - \frac{{f(m({X^{(0)}}))f(m({Y^{(0)}}))f(m({Z^{(0)}}))\left( {{f^2}(m({X^{(0)}})) + f(m({Y^{(0)}}))(f(m({Y^{(0)}})) - f(m({Z^{(0)}})))} \right)}}{{{{(f(m({X^{(0)}})) - f(m({Z^{(0)}})))}^2}{{(f(m({X^{(0)}})) - f(m({Y^{(0)}})))}^2}(f(m({Y^{(0)}})) - f(m({Z^{(0)}})))}}\frac{{f(m({X^{(0)}}))}}{{f'(m({X^{(0)}}))}}.
\end{array}\]
\end{scriptsize}
It follows that
\begin{scriptsize}
\[\begin{array}{l}
m({X^{(1)}}) \in \left\{ {m({Z^{(0)}}) - \frac{{f(m({X^{(0)}}))f(m({Y^{(0)}}))f(m({Z^{(0)}}))\left( {{f^2}(m({X^{(0)}})) + f(m({Y^{(0)}}))(f(m({Y^{(0)}})) - f(m({Z^{(0)}})))} \right)}}{{{{(f(m({X^{(0)}})) - f(m({Z^{(0)}})))}^2}{{(f(m({X^{(0)}})) - f(m({Y^{(0)}})))}^2}(f(m({Y^{(0)}})) - f(m({Z^{(0)}})))}}\frac{{f(m({X^{(0)}}))}}{{F'({X^{(0)}})}}} \right\}.
\end{array}\]
\end{scriptsize}
Let
\begin{equation}\label{4-8}
\begin{scriptsize}
\begin{array}{l}
T({X^{(0)}},{Y^{(0)}},{Z^{(0)}}) = {m({Z^{(0)}}) - \frac{{f(m({X^{(0)}}))f(m({Y^{(0)}}))f(m({Z^{(0)}}))\left( {{f^2}(m({X^{(0)}})) + f(m({Y^{(0)}}))(f(m({Y^{(0)}})) - f(m({Z^{(0)}})))} \right)}}{{{{(f(m({X^{(0)}})) - f(m({Z^{(0)}})))}^2}{{(f(m({X^{(0)}})) - f(m({Y^{(0)}})))}^2}(f(m({Y^{(0)}})) - f(m({Z^{(0)}})))}}\frac{{f(m({X^{(0)}}))}}{{F'({X^{(0)}})}}},
\end{array}
\end{scriptsize}
\end{equation}
therefore $ m({X^{(1)}}) \in T({X^{(0)}}, {Y^{(0)}}, {Z^{(0)}}) $. Since $ m({X^{(1)}}) \in {X^{(1)}} \subseteq Z^{(0)} $, we have
\[\begin{array}{l}
m({X^{(1)}})\in T({X^{(0)}},{Y^{(0)}},{Z^{(0)}}) \cap Z^{(0)}.
\end{array}\]
Define
\begin{equation}\label{4-9}
{X^{(1)}} = T({X^{(0)}},{Y^{(0)}},{Z^{(0)}}) \cap Z^{(0)}.
\end{equation}
Now by continuing this process, we see that
\begin{equation}\label{5}
\begin{cases}
{Y^{(k)}} = N({X^{(k)}}) \cap {X^{(k)}},\\
{Z^{(k)}} = K({X^{(k)}}, {Y^{(k)}}) \cap {Y^{(k)}},\\
{X^{(k + 1)}} = T({X^{(k)}}, {Y^{(k)}}, {Z^{(k)}}) \cap {Z^{(k)}},
\end{cases}
\end{equation}
where
\begin{equation}\label{5-1}
N({X^{(k)}}) = {m({X^{(k)}}) - \dfrac{{f(m({X^{(k)}}))}}{{F'({X^{(k)}})}}},
\end{equation}
\begin{equation}\label{5-2}
K({X^{(k)}}, {Y^{(k)}}) = {m({Y^{(k)}}) - \dfrac{{f(m({X^{(k)}}))f(m({Y^{(k)}}))}}{{{{(f(m({X^{(k)}})) - f(m({Y^{(k)}})))}^2}}}\dfrac{{f(m({X^{(k)}}))}}{{F'({X^{(k)}})}}},
\end{equation}
and
\begin{equation}\label{5-3}
\begin{small}
\begin{array}{l}
T({X^{(k)}}, {Y^{(k)}}, {Z^{(k)}}) = {m({Z^{(k)}}) - \left( {\left( {f(m({X^{(k)}}))f(m({Y^{(k)}}))f(m({Z^{(k)}}))\left( {f{{(m({X^{(k)}}))}^2}} \right.} \right.} \right.}\\
~~~~~~~~~~~~~~~~~~~~~~~~\left. {\left. { + f(m({Y^{(k)}}))\left( {f(m({Y^{(k)}})) - f(m({Z^{(k)}}))} \right)} \right)} \right)/\left( {\left( {f(m({X^{(k)}}))} \right.} \right.\\
~~~~~~~~~~~~~~~~~~~~~~~~~{\left. { - f(m({Y^{(k)}}))} \right)^2}{\left( {f(m({X^{(k)}})) - f(m({Z^{(k)}}))} \right)^2}\left( {f(m({Y^{(k)}}))} \right.\\
~~~~~~~~~~~~~~~~~~~~~~~~ {\left. {\left. {\left. { - f(m({Z^{(k)}}))} \right)} \right)} \right)\left( {f(m({X^{(k)}}))/F'({X^{(k)}})} \right)}.
\end{array}
\end{small}
\end{equation}
Thus, an interval extension of the three-step Kung and Traub's iterative Method is produced.


\begin{theorem}\label{f6}
Let $ f\in C({X^{(0)}}) $ and $ 0 \notin F'( {{X^{(k)}}} ) $ for $ k=0, 1, 2, \cdots . $ If an interval $ {X^{(0)}} $ contains a root $ x^* $ of $ f, $ then so do intervals $ {X^{(k)}}, k = 1, 2, \cdots . $ Besides, the nested interval sequence $ \{X^{(k)}\} $ of the form \eqref{5} converging to $ {x^*} $.
\end{theorem}

\begin{proof}
By induction, since $ 0\notin F'( {{X^{(k)}}} ), $ if $ {x^*}\in {X^{(0)}} $ then $ {x^*}\in {X^{(k)}} $ for $ k = 1, 2, \cdots $. Also, according to Lemma \ref{f3} and Lemma \ref{f4}, since $ \{X^{(k)}\} $ is nested interval sequence of the form \eqref{5}, and $ {x^*}\in {X^{(k)}} $ for $ k = 0, 1, 2, \cdots, $ therefore $ {x^*}\in \bigcap\limits_k {{X^{(k)}}} $ or $ \mathop {\lim }\limits_{n \to \infty } \bigcap\limits_{k = 0}^n {{X^{(k)}}}  = {x^*} $ and the proof is completed.
\end{proof}


\begin{theorem}\label{f7}
Let $ f\in C ({X^{(0)}}) $ and $ 0 \notin F'( {{X^{(k)}}} ) $ for $ k=0, 1, 2, \cdots . $
\begin{itemize}
\item[(i)] The iteration \eqref{5} stops after finitely many steps with empty $ {X^{(k)}} = \emptyset $ if and only if $ f $ has no zero in $ {X^{(0)}} $.
\item[(ii)] If $ K({X^{(k)}}, {Y^{(k)}}) \subset {Y^{(k)}} $ and $ T({X^{(k)}}, {Y^{(k)}}, {Z^{(k)}}) \subset {Z^{(k)}} $, then $ {X^{(k)}} $ contains exactly one root of $ f $.\\
In this case,
\begin{equation}\label{6}
w( {{X^{( k + 1 )}}} ) \leq \gamma \, ( w( {{X^{( k )}}} ))^4 .
\end{equation}
\end{itemize}
\end{theorem}

\begin{proof}
$ (i) $ Assume that $ {X^{(0)}} $ contains a root $ {x^*} $, then Theorem \ref{f4-1} results $ {x^*}\in {Y^{(k)}} $, therefore from Theorem \ref{f6} we have $ {x^*}\in K({X^{(k)}}, {Y^{(k)}}) $ which means that $ {x^*}\in K({X^{(k)}}, {Y^{(k)}}) \cap {Y^{(k)}} = {Z^{(k)}} $. Similarly, Theorem \ref{f6} results $ {x^*}\in T({X^{(k)}}, {Y^{(k)}}, {Z^{(k)}}) $, so we get $ {x^*}\in T({X^{(k)}}, {Y^{(k)}}, {Z^{(k)}}) \cap {Z^{(k)}} $. Therefore, if $ T({X^{(k)}}, {Y^{(k)}}, {Z^{(k)}}) \cap {Z^{(k)}} = \emptyset $, then $ {X^{(0)}} $ cannot contain a root of $ f $. For the converse, assume that $ f $ has no zero in $ {X^{(0)}} $. Since $ \{X^{(k)}\} $ is nested interval sequence, it is clear that $ {X^{(k)}} = \emptyset, $  for $ k=0, 1, 2, \cdots . $
\\
$ (ii) $ Since $ 0 \notin F'( {{X^{(k)}}} ) $, then $ f'( x ) \neq 0 $ for all $ x\in {X^{(k)}} $ and $ f $ is monotonic on $ {X^{(k)}} $. Therefore, since $ f $ is continuous on $ {X^{(0)}} $, there can be at most one root in $ {X^{(0)}} $. In other words, it has at most one zero in $ {X^{(k)}} $. Hence, it is sufficient to find a zero $ x^* \in {X^{(k)}} $. Using the Theorem \ref{f6} it is clear that $ f $ has exactly one root in $ {X^{(k)}} $.
\\
Now we want to prove \eqref{6}. Since $ K({X^{(k)}}, {Y^{(k)}}) \subset {Y^{(k)}} $, thus from \eqref{5}, we get
\[{Z^{(k)}} = {m({Y^{(k)}}) - \dfrac{{f(m({X^{(k)}}))f(m({Y^{(k)}}))}}{{{{(f(m({X^{(k)}})) - f(m({Y^{(k)}})))}^2}}}\dfrac{{f(m({X^{(k)}}))}}{{F'({X^{(k)}})}}}.\]
It is clear that
\begin{equation}\label{7}
\begin{array}{l}
{Z^{(k)}} = m({Y^{(k)}}) - \dfrac{{f(m({Y^{(k)}}))}}{{{{\left( {1 - \left( {f(m({Y^{(k)}}))/f(m({X^{(k)}}))} \right)} \right)}^2}}}\dfrac{1}{{F'({X^{(k)}})}}.
\end{array}
\end{equation}
Let
\begin{equation}\label{11}
\left| {{{\left( {1 - \left( {f(m({Y^{(k)}}))/f(m({X^{(k)}}))} \right)} \right)}^2}} \right| \geq {K_1},
\end{equation}
From \eqref{7} and \eqref{11} we obtain
\begin{equation}\label{9-1}
\begin{array}{l}
w({Z^{(k)}}) = \dfrac{{|f(m({Y^{(k)}}))|}}{{{{\left| {1 - \left( {f(m({Y^{(k)}}))/f(m({X^{(k)}}))} \right)} \right|}^2}}}w\left( {\dfrac{1}{{F'({X^{(k)}})}}} \right)\\
~~~~~~~~~~ \leq \dfrac{{|f(m({Y^{(k)}}))|}}{{{K_1}}}w\left( {\dfrac{1}{{F'({X^{(k)}})}}} \right).
\end{array}
\end{equation}
Since $ {x^*} $ is a simple root of $ f, $ we can write
\[f(m({Y^{(k)}})) = f'(\xi)(m({Y^{(k)}}) - {x^*}),\]
where $ \xi $ is between $ m({Y^{(k)}}) $ and $ x^* $. Let $ |f'(\xi)|\leq K_2 $, since $ Y^{(k)} $ is generated from \eqref{1}, Theorem \ref{f5} leads to
\[|f(m({Y^{(k)}}))| = |f'(\xi)|\, |(m({Y^{(k)}}) - {x^*})|\leq {K_2} w({Y^{(k)}}) \leq {K_2}C \left( {w({X^{(k)}})} \right)^2 .\]
Also, from Definition \ref{f1} we have
\begin{equation}\label{12-11}
w\left(\dfrac{1}{F'(X^{(k)})}\right) \leq L_1 w({X^{(k)}}),
\end{equation}
therefore, we obtain
\begin{equation}\label{12-2}
w({Z^{(k)}}) \leq \dfrac{{L_1}{K_2}C}{{{K_1}}}{( {w({X^{(k)}})})^3}.
\end{equation}
Since $ T({X^{(k)}}, {Y^{(k)}}, {Z^{(k)}}) \subset {Z^{(k)}} $, thus from \eqref{5}, we get
\begin{equation}\label{9-2}
\begin{scriptsize}
\begin{array}{l}
{X^{(k + 1)}} = m({Z^{(k)}}) - \dfrac{{f(m({X^{(k)}}))f(m({Y^{(k)}}))f(m({Z^{(k)}}))\left( {f{{(m({X^{(k)}}))}^2} + f(m({Y^{(k)}}))\left( {f(m({Y^{(k)}})) - f(m({Z^{(k)}}))} \right)} \right)}}{{{{\left( {f(m({X^{(k)}})) - f(m({Y^{(k)}}))} \right)}^2}{{\left( {f(m({X^{(k)}})) - f(m({Z^{(k)}}))} \right)}^2}\left( {f(m({Y^{(k)}})) - f(m({Z^{(k)}}))} \right)}}\dfrac{{f(m({X^{(k)}}))}}{{F'({X^{(k)}})}}.
\end{array}
\end{scriptsize}
\end{equation}
It is clear that
\begin{equation}\label{9-4}
\begin{scriptsize}
\begin{array}{l}
{X^{(k + 1)}} = m({Z^{(k)}}) - \dfrac{{f(m({Z^{(k)}}))\left( {1 + \left( {f(m({Y^{(k)}}))/f{{(m({X^{(k)}}))}^2}} \right)\left( {f(m({Y^{(k)}})) - f(m({Z^{(k)}}))} \right)} \right)}}{{{{\left( {1 - \left( {f(m({Y^{(k)}}))/f(m({X^{(k)}}))} \right)} \right)}^2}{{\left( {1 - \left( {f(m({Z^{(k)}}))/f(m({X^{(k)}}))} \right)} \right)}^2}\left( {1 - \left( {f(m({Z^{(k)}}))/f(m({Y^{(k)}}))} \right)} \right)}}\dfrac{1}{{F'({X^{(k)}})}}
\end{array}
\end{scriptsize}
\end{equation}
Let
\begin{equation}\label{9-3}
\begin{scriptsize}
\begin{array}{l}
\left| {\dfrac{{\left( {1 + \left( {f(m({Y^{(k)}}))/f{{(m({X^{(k)}}))}^2}} \right)\left( {f(m({Y^{(k)}})) - f(m({Z^{(k)}}))} \right)} \right)}}{{{{\left( {1 - \left( {f(m({Y^{(k)}}))/f(m({X^{(k)}}))} \right)} \right)}^2}{{\left( {1 - \left( {f(m({Z^{(k)}}))/f(m({X^{(k)}}))} \right)} \right)}^2}\left( {1 - \left( {f(m({Z^{(k)}}))/f(m({Y^{(k)}}))} \right)} \right)}}} \right| \leq {K_3}
\end{array}
\end{scriptsize}
\end{equation}
From \eqref{9-4} and \eqref{9-3} we obtain
\begin{equation}\label{9-5}
\begin{scriptsize}
\begin{array}{l}
w({X^{(k + 1)}}) = \left| {\frac{{\left( {1 + \left( {f(m({Y^{(k)}}))/f{{(m({X^{(k)}}))}^2}} \right)\left( {f(m({Y^{(k)}})) - f(m({Z^{(k)}}))} \right)} \right)}}{{{{\left( {1 - \left( {f(m({Y^{(k)}}))/f(m({X^{(k)}}))} \right)} \right)}^2}{{\left( {1 - \left( {f(m({Z^{(k)}}))/f(m({X^{(k)}}))} \right)} \right)}^2}\left( {1 - \left( {f(m({Z^{(k)}}))/f(m({Y^{(k)}}))} \right)} \right)}}} \right||f(m({X^{(k)}}))|w\left( {\dfrac{1}{{F'({X^{(k)}})}}} \right)\\
~~~~~~~~~~~~~~ \leq {K_3}|f(m({X^{(k)}}))|w\left( {\frac{1}{{F'({X^{(k)}})}}} \right)
\end{array}
\end{scriptsize}
\end{equation}
Since $ {x^*} $ is a simple root of $ f, $ we can write
\begin{equation}\label{8-3}
f(m({Z^{(k)}})) = f'(\eta)(m({Z^{(k)}}) - {x^*}),
\end{equation}
where $ \eta $ is between $ m({Z^{(k)}}) $ and $ x^* $.  It is clear that
\begin{equation}\label{8-4}
|m({Z^{(k)}}) - {x^*}| \leq w({Z^{(k)}}).
\end{equation}
Let $ |f'({\eta})| \leq {K_4} $. From \eqref{12-11}, \eqref{12-2}, \eqref{9-5}, \eqref{8-3} and \eqref{8-4} we get
\[w({X^{(k + 1)}}) \leq \dfrac{{L_1^2{K_2}{K_3}{K_4}C}}{{{K_1}}}{(w({X^{(k)}}))^4}.\]
where
\[ \gamma = \dfrac{{L_1^2{K_2}{K_3}{K_4}C}}{{{K_1}}}. \]
\end{proof}



\section{Numerical results}\label{3}
In this section we report computational results. The list of test nonlinear functions are presented in Table \ref{tab:1}.
In Tables \ref{tab:2}, \ref{tab:3}, \ref{tab:4}, \ref{tab:5}, \ref{tab:6}, interval three-step Kung and Traub method \eqref{5} is compared with interval version of Traub's three-step method \eqref{2-1-1}, interval King method \eqref{2-1}, (with $ \beta =2 $), interval Ostrowski method \eqref{2-1}, (with $ \beta =0 $) and interval Newton method \eqref{1}. All methods are computed by using INTLAB toolbox created by Rump \cite{R}.\\
\begin{table}[h!]
\caption{Tested functions and initial intervals}\label{tab:1}
\resizebox{\textwidth}{!}{\begin{tabular}{lllll}
\hline\noalign{\smallskip}
\multicolumn{1}{l}{Example $ i $} & \multicolumn{1}{l}{Function $ f_i $} & \multicolumn{1}{l}{Root $ x^* $} & \multicolumn{1}{l}{Initial interval $ X^{(0)} $} & \multicolumn{1}{l}{Root enclosures}\\
\noalign{\smallskip}\hline\noalign{\smallskip}
\multicolumn{1}{l}{1} & \multicolumn{1}{l}{$ {x^3} + \sin \left( {\frac{x}{{\sqrt 3 }}} \right) - \frac{1}{4} $} & \multicolumn{1}{l}{$ 0.3568342187225045 $} & \multicolumn{1}{l}{$ [0,0.8] $} & \multicolumn{1}{l}{$ [0.35683421872250, 0.35683421872251] $ } \\
\multicolumn{1}{l}{2} & \multicolumn{1}{l}{$ \cos x + x - x^2 + x^5 $} & \multicolumn{1}{l}{$ -0.5333964635678204 $} & \multicolumn{1}{l}{$ [-0.9,-0.2] $} & \multicolumn{1}{l}{$ [-0.53339646356783, -0.53339646356782] $} \\
\multicolumn{1}{l}{3} & \multicolumn{1}{l}{$ {e^x} - {\sin ^3 x} $} & \multicolumn{1}{l}{$ -3.4623979938206757 $} & \multicolumn{1}{l}{$ [-3.5,-3.4] $} & \multicolumn{1}{l}{$ [-3.46239799382068, -3.46239799382067] $} \\
\multicolumn{1}{l}{4} & \multicolumn{1}{l}{$ (x - 1){e^{ - 2x}} + {x^3} $} & \multicolumn{1}{l}{$ 0.5391809932576055 $} & \multicolumn{1}{l}{$ [0.4,0.6] $} & \multicolumn{1}{l}{$ [0.53918099325760, 0.53918099325761] $ } \\
\multicolumn{1}{l}{5} & \multicolumn{1}{l}{$ {\sin ^2}({x^2}+1) - \frac{\sqrt{x+1}}{3} $} & \multicolumn{1}{l}{$ 1.1684762578039694 $} & \multicolumn{1}{l}{$ [1,1.2] $} & \multicolumn{1}{l}{$ [1.16847625780396, 1.16847625780397] $} \\
\noalign{\smallskip}\hline
\end{tabular}}
\end{table}
\begin{table}[h!]
\caption{Comparison of $ w({X^{(k)}}) $ for example $ {f_1}(x) $}\label{tab:2}
\resizebox{\textwidth}{!}{\begin{tabular}{cccccc}
\hline\noalign{\smallskip}
\multicolumn{1}{c}{Number of iterations $ k $} & \multicolumn{5}{c}{Methods} \\ \cline{2-6} & \multicolumn{1}{c}{\eqref{1}} & \multicolumn{1}{c}{\eqref{2-1}, $ \beta =0 $} & \multicolumn{1}{c}{\eqref{2-1}, $ \beta =2 $} & \multicolumn{1}{c}{\eqref{2-1-1}} & \multicolumn{1}{c}{\eqref{5}} \\
\noalign{\smallskip}\hline\noalign{\smallskip}
\multicolumn{1}{c}{1} & \multicolumn{1}{c}{$ 6.58\times 10^{-2} $} & \multicolumn{1}{c}{$ 7.65\times 10^{-3} $} & \multicolumn{1}{c}{$ 6.95\times 10^{-3} $} & \multicolumn{1}{c}{$ 9.81\times 10^{-4} $} & \multicolumn{1}{c}{$ 1.53\times 10^{-3} $} \\
\multicolumn{1}{c}{2} & \multicolumn{1}{c}{$ 1.06\times 10^{-3} $} & \multicolumn{1}{c}{$ 2.39\times 10^{-8} $} & \multicolumn{1}{c}{$ 4.76\times 10^{-8} $} & \multicolumn{1}{c}{$ 2.33\times 10^{-15} $} & \multicolumn{1}{c}{$ 2.22\times 10^{-16} $} \\
\multicolumn{1}{c}{3} & \multicolumn{1}{c}{$ 1.73 \times 10^{-7} $} & \multicolumn{1}{c}{$ 2.22\times 10^{-16} $} & \multicolumn{1}{c}{$ 2.22\times 10^{-16} $} & \multicolumn{1}{c}{$ 2.22\times 10^{-16} $} & \multicolumn{1}{c}{} \\
\multicolumn{1}{c}{4} & \multicolumn{1}{c}{$ 2.39 \times 10^{-15} $}
& \multicolumn{1}{c}{} & \multicolumn{1}{c}{} & \multicolumn{1}{c}{} & \multicolumn{1}{c}{} \\
\multicolumn{1}{c}{5} & \multicolumn{1}{c}{$ 2.22 \times 10^{-16} $}
& \multicolumn{1}{c}{} & \multicolumn{1}{c}{} & \multicolumn{1}{c}{} & \multicolumn{1}{c}{} \\
\noalign{\smallskip}\hline
\end{tabular}}
\end{table}\\
\begin{table}[h!]
\caption{Comparison of $ w({X^{(k)}}) $ for example $ {f_2}(x) $}\label{tab:3}
\resizebox{\textwidth}{!}{\begin{tabular}{cccccc}
\hline\noalign{\smallskip}
\multicolumn{1}{c}{Number of iterations $ k $} & \multicolumn{5}{c}{Methods} \\ \cline{2-6} & \multicolumn{1}{c}{\eqref{1}} & \multicolumn{1}{c}{\eqref{2-1}, $ \beta =0 $} & \multicolumn{1}{c}{\eqref{2-1}, $ \beta =2 $} & \multicolumn{1}{c}{\eqref{2-1-1}} & \multicolumn{1}{c}{\eqref{5}} \\
\noalign{\smallskip}\hline\noalign{\smallskip}
\multicolumn{1}{c}{1} & \multicolumn{1}{c}{$ 2.4\times 10^{-2} $} & \multicolumn{1}{c}{$ 2.91\times 10^{-3} $} & \multicolumn{1}{c}{$ 2.61\times 10^{-3} $} & \multicolumn{1}{c}{$ 5.46\times 10^{-4} $} & \multicolumn{1}{c}{$ 5.01\times 10^{-4} $} \\
\multicolumn{1}{c}{2} & \multicolumn{1}{c}{$ 1.29\times 10^{-4} $} & \multicolumn{1}{c}{$ 7.81\times 10^{-10} $} & \multicolumn{1}{c}{$ 1.91\times 10^{-9} $} & \multicolumn{1}{c}{$ 5.55\times 10^{-16} $} & \multicolumn{1}{c}{$ 2.22\times 10^{-16} $} \\
\multicolumn{1}{c}{3} & \multicolumn{1}{c}{$ 1.97\times 10^{-9} $} & \multicolumn{1}{c}{$ 2.22\times 10^{-16} $} & \multicolumn{1}{c}{$ 2.22\times 10^{-16} $} & \multicolumn{1}{c}{} & \multicolumn{1}{c}{} \\
\multicolumn{1}{c}{4} & \multicolumn{1}{c}{$ 2.22 \times 10^{-16} $}
& \multicolumn{1}{c}{} & \multicolumn{1}{c}{} & \multicolumn{1}{c}{} & \multicolumn{1}{c}{} \\
\noalign{\smallskip}\hline
\end{tabular}}
\end{table}\\
\begin{table}[h!]
\caption{Comparison of $ w({X^{(k)}}) $ for example $ {f_3}(x) $}\label{tab:4}
\resizebox{\textwidth}{!}{\begin{tabular}{cccccc}
\hline\noalign{\smallskip}
\multicolumn{1}{c}{Number of iterations $ k $} & \multicolumn{5}{c}{Methods} \\ \cline{2-6} & \multicolumn{1}{c}{\eqref{1}} & \multicolumn{1}{c}{\eqref{2-1}, $ \beta =0 $} & \multicolumn{1}{c}{\eqref{2-1}, $ \beta =2 $} & \multicolumn{1}{c}{\eqref{2-1-1}} & \multicolumn{1}{c}{\eqref{5}} \\
\noalign{\smallskip}\hline\noalign{\smallskip}
\multicolumn{1}{c}{1} & \multicolumn{1}{c}{$ 8\times 10^{-3} $} & \multicolumn{1}{c}{$ 7.07\times 10^{-4} $} & \multicolumn{1}{c}{$ 6.74\times 10^{-4} $} & \multicolumn{1}{c}{$ 1.22\times 10^{-4} $} & \multicolumn{1}{c}{$ 5.08\times 10^{-5} $} \\
\multicolumn{1}{c}{2} & \multicolumn{1}{c}{$ 5.94\times 10^{-5} $} & \multicolumn{1}{c}{$ 5.89\times 10^{-11} $} & \multicolumn{1}{c}{$ 1.69\times 10^{-10} $} & \multicolumn{1}{c}{$ 8.88\times 10^{-16} $} & \multicolumn{1}{c}{$ 8.88\times 10^{-16} $} \\
\multicolumn{1}{c}{3} & \multicolumn{1}{c}{$ 1.25\times 10^{-9} $} & \multicolumn{1}{c}{$ 8.88\times 10^{-16} $} & \multicolumn{1}{c}{$ 8.88\times 10^{-16} $} & \multicolumn{1}{c}{} & \multicolumn{1}{c}{} \\
\multicolumn{1}{c}{4} & \multicolumn{1}{c}{$ 8.88\times 10^{-16} $}
& \multicolumn{1}{c}{} & \multicolumn{1}{c}{} & \multicolumn{1}{c}{} & \multicolumn{1}{c}{} \\
\noalign{\smallskip}\hline
\end{tabular}}
\end{table}\\
\begin{table}[h!]
\caption{Comparison of $ w({X^{(k)}}) $ for example $ {f_4}(x) $}\label{tab:5}
\resizebox{\textwidth}{!}{\begin{tabular}{cccccc}
\hline\noalign{\smallskip}
\multicolumn{1}{c}{Number of iterations $ k $} & \multicolumn{5}{c}{Methods} \\ \cline{2-6} & \multicolumn{1}{c}{\eqref{1}} & \multicolumn{1}{c}{\eqref{2-1}, $ \beta =0 $} & \multicolumn{1}{c}{\eqref{2-1}, $ \beta =2 $} & \multicolumn{1}{c}{\eqref{2-1-1}} & \multicolumn{1}{c}{\eqref{5}} \\
\noalign{\smallskip}\hline\noalign{\smallskip}
\multicolumn{1}{c}{1} & \multicolumn{1}{c}{$ 2.92\times 10^{-2} $} & \multicolumn{1}{c}{$ 2.45\times 10^{-3} $} & \multicolumn{1}{c}{$ 2.35\times 10^{-3} $} & \multicolumn{1}{c}{$ 3.44\times 10^{-4} $} & \multicolumn{1}{c}{$ 2.06\times 10^{-4} $} \\
\multicolumn{1}{c}{2} & \multicolumn{1}{c}{$ 3.88\times 10^{-4} $} & \multicolumn{1}{c}{$ 2.34\times 10^{-10} $} & \multicolumn{1}{c}{$ 5.55\times 10^{-10} $} & \multicolumn{1}{c}{$ 2.22\times 10^{-16} $} & \multicolumn{1}{c}{$ 2.22\times 10^{-16} $} \\
\multicolumn{1}{c}{3} & \multicolumn{1}{c}{$ 3.51\times 10^{-10} $} & \multicolumn{1}{c}{$ 2.22\times 10^{-16} $} & \multicolumn{1}{c}{$ 2.22\times 10^{-16} $} & \multicolumn{1}{c}{} & \multicolumn{1}{c}{} \\
\multicolumn{1}{c}{4} & \multicolumn{1}{c}{$ 2.22 \times 10^{-16} $}
& \multicolumn{1}{c}{} & \multicolumn{1}{c}{} & \multicolumn{1}{c}{} & \multicolumn{1}{c}{} \\
\noalign{\smallskip}\hline
\end{tabular}}
\end{table}\\
\begin{table}[h!]
\caption{Comparison of $ w({X^{(k)}}) $ for example $ {f_5}(x) $}\label{tab:6}
\resizebox{\textwidth}{!}{\begin{tabular}{cccccc}
\hline\noalign{\smallskip}
\multicolumn{1}{c}{Number of iterations $ k $} & \multicolumn{5}{c}{Methods} \\ \cline{2-6} & \multicolumn{1}{c}{\eqref{1}} & \multicolumn{1}{c}{\eqref{2-1}, $ \beta =0 $} & \multicolumn{1}{c}{\eqref{2-1}, $ \beta =2 $} & \multicolumn{1}{c}{\eqref{2-1-1}} & \multicolumn{1}{c}{\eqref{5}} \\
\noalign{\smallskip}\hline\noalign{\smallskip}
\multicolumn{1}{c}{1} & \multicolumn{1}{c}{$ 5.33\times 10^{-2} $} & \multicolumn{1}{c}{$ 5.74\times 10^{-3} $} & \multicolumn{1}{c}{$ 5.61\times 10^{-3} $} & \multicolumn{1}{c}{$ 1.95\times 10^{-3} $} & \multicolumn{1}{c}{$ 1\times 10^{-3} $} \\
\multicolumn{1}{c}{2} & \multicolumn{1}{c}{$ 1.39\times 10^{-3} $} & \multicolumn{1}{c}{$ 1.78\times 10^{-8} $} & \multicolumn{1}{c}{$ 1.37\times 10^{-8} $} & \multicolumn{1}{c}{$ 1.64\times 10^{-13} $} & \multicolumn{1}{c}{$ 6.66\times 10^{-16} $} \\
\multicolumn{1}{c}{3} & \multicolumn{1}{c}{$ 5.67 \times 10^{-7} $} & \multicolumn{1}{c}{$ 6.66\times 10^{-16} $} & \multicolumn{1}{c}{$ 6.66\times 10^{-16} $} & \multicolumn{1}{c}{$ 4.44\times 10^{-16} $} & \multicolumn{1}{c}{} \\
\multicolumn{1}{c}{4} & \multicolumn{1}{c}{$ 9.77 \times 10^{-15} $}
& \multicolumn{1}{c}{} & \multicolumn{1}{c}{} & \multicolumn{1}{c}{} & \multicolumn{1}{c}{} \\
\multicolumn{1}{c}{5} & \multicolumn{1}{c}{$ 6.66 \times 10^{-16} $}
& \multicolumn{1}{c}{} & \multicolumn{1}{c}{} & \multicolumn{1}{c}{} & \multicolumn{1}{c}{} \\
\noalign{\smallskip}\hline
\end{tabular}}
\end{table}\\
Numerical results show that the interval extension of the three-step Kung and Traub's method is faster than the existing methods.


\section{Conclusion}\label{sec:4}


In this paper, an interval extension of the three-step Kung and Traub's method which calculates the enclosure solutions of a given nonlinear equation is produced. Also, error bound and convergence rate were studied. This algorithm is tested using some examples via INTLAB. Numerical results show that the interval extension of the three-step Kung and Traub's method is better than the existing methods.


\end{document}